\documentclass[10pt]{amsart}

\usepackage{hyperref}

\usepackage{amssymb,latexsym} 
\usepackage{palatino}
\usepackage{mathrsfs} 
\usepackage{amscd}
\usepackage{amsmath}

\newcommand{\g}{\mathfrak{g}}

\newcommand{\complex}{\mathbf C}

\newcommand{\zz}{\mathbf Z}
\newcommand{\qq}{\mathbf Q}

\newcommand{\ro}{{\Phi}}

\newcommand{\orbit}{\mathcal O}
\newcommand{\levi}{\mathfrak l}

\renewcommand{\wr}{\hat{W}}
\newcommand{\ar}{{\hat{A}}(e)}
\newcommand{\F}{\mathbb F}

\renewcommand{\H}{{\mathcal{H}_t}}
\newcommand{\flag}{{\mathcal B}}

\newcommand{\dualp}{\lambda'}

\newcommand{\rep}{\chi^{\alpha, \beta}}
\newcommand{\ttau}{\tilde{\tau}}

\DeclareMathOperator{\Ind}{Ind}

\theoremstyle{plain}
\newtheorem{theorem}{Theorem}
\newtheorem{lemma}[theorem]{Lemma}
\newtheorem{conjecture}[theorem]{Conjecture}
\newtheorem{corollary}[theorem]{Corollary}
\newtheorem{proposition}[theorem]{Proposition}

\theoremstyle{definition}

\theoremstyle{remark}
\newtheorem{remark}[theorem]{Remark}

\title[Exterior powers]{Exterior powers of the reflection representation in Springer theory}

\author{Eric Sommers}
\address{Department of Mathematics and Statistics\\
University of Massachusetts---Amherst\\ 
Amherst, MA 01003}
\email{esommers@math.umass.edu}

\begin{document}

\begin{abstract}
Let $H^*(\flag_e)$ be the total Springer representation of $W$ for the nilpotent element $e$ in a simple Lie algebra $\g$.   
Let $\wedge^i V$ denote the $i$th exterior power of the reflection representation $V$ of $W$.
The focus of this paper is on the algebra of $W$-invariants in 
$$H^*(\flag_e) \otimes \wedge^*V$$ and we show that it is an exterior algebra on the subspace 
$(H^*(\flag_e) \otimes V)^W$ in some cases that were not previously known.
This result was established for $e=0$ by Solomon \cite{solomon} and was proved
by Henderson \cite{henderson:exterior} in types $A,B, C$  
when $e$ is regular in a Levi subalgebra.   

The above statement about the $W$-invariants implies a conjecture of Lehrer-Shoji \cite{lehrer-shoji:reflections} about the occurrences of $\wedge^i V$ in $H^*(\flag_e)$, which was originally stated when $e$ is regular in a Levi subalgebra.   
In this paper we prove the Lehrer-Shoji conjecture in all types and its natural extension to any nilpotent $e$, not only those that are regular in a Levi subalgebra.

In the last part of the paper we make a connection to rational Cherednik algebras and this leads to an explanation for the appearance in Springer theory of the Orlik-Solomon exponents coming from hyperplane arrangements.  This connection was established in the classical groups in \cite{lehrer-shoji:reflections}, \cite{spaltenstein:reflection} after being observed empirically by Orlik, Solomon, and Spaltenstein in the exceptional groups.
\end{abstract}

\thanks{I thank Anthony Henderson and Vic Reiner for helpful conversations and for papers that motivated the present work. 
I also thank Kyoji Saito and Jean Michel.  Supported by NSF grant DMS-0201826.}

\maketitle

\section{Introduction}

Let $G$ be a simple algebraic group over the algebraically closed field $\mathbf{k}$ of good characteristic $p$.  Let $\g$ be its Lie algebra.
Assume $G$ is of adjoint type and connected.
Let $n$ be the rank of $G$ and $V$ the reflection representation of the Weyl group $W$, which is irreducible of dimension $n$.

For a nilpotent element $e \in \g$, consider the $l$-adic cohomology $H^*(\flag_e)$ of the Springer fiber $\flag_e$ with coefficients in $\overline{\qq}_l$, 
an algebraic closure of the $l$-adic numbers, with $l \neq p$.
The cohomology carries a representation of $W$ defined by Springer \cite{springer:green}.  We follow the definition where $H^0(\flag_e)$ carries
the trivial representation of $W$ \cite{hotta:springer}, \cite{lusztig:green}.

This paper is concerned with studying the $W$-invariants in the bi-graded algebra  $$H^*(\flag_e) \otimes \wedge^* V.$$
Except for two cases it seems that the $W$-invariants are themselves an exterior algebra and we conjecture:
\begin{conjecture}
\label{conj1}
%Let $e$ be a nilpotent element in $\g$.
Except when $e = F_4(a_3)$ in $F_4$ or $e = E_8(a_7)$ in $E_8$, 
the algebra $$\displaystyle ( \bigoplus_{i=0}^n H^*(\flag_e) \otimes \wedge^i V)^W$$ is an exterior algebra on the 
subspace $(H^*(\flag_e) \otimes V)^W.$
\end{conjecture}

This is known to be true for $e=0$ by Solomon \cite{solomon} since in that case $\flag_e = \flag$ and $H^*(\flag)$ is isomorphic to the coinvariant algebra of $V$.
For $e$ regular in a Levi subalgebra in types $A_n,B_n,C_n$, Henderson \cite{henderson:exterior} has shown that the conjecture is true.   That the conjecture cannot be
true for the two exceptional cases mentioned in the conjecture will be seen shortly.  
Nonetheless, it does appear in those two cases
that the $W$-invariants are isomorphic to a quotient of the exterior algebra on  $(H^*(\flag_e) \otimes V)^W$.

%If $H$ is a finite group, we denote by $\hat{H}$ the set of irreducible complex characters of $H$ and $R_H$ the character ring of $H$
%over the ring $\zz[q]$.  Let $\langle \cdot , \cdot \rangle$ be the inner product on $R_H$.

%\subsection{} 
For $H$ a finite group, let $\hat{H}$ be the set of irreducible finite-dimensional 
representations of $H$ over $\overline{\qq}_l$. 
%, an algebraic closure of the $l$-adic numbers with $l \neq p$ .
Let $R(H)$ be the Grothendieck group of finite-dimensional representations of $H$ over $\overline{\qq}_l$ with 
coefficients in the ring $\qq(q)$ of rational functions in a variable $q$.  Let $\langle \cdot , \cdot \rangle_H$ be the inner product on $R(H)$.
We suppress the subscript if there is no ambiguity about which $H$ we are talking about.

% and denote by $\ar$ the set of irreducible complex characters of $A(e)$ and
%$R_{A(e)}$ the character ring of $A(e)$ over the ring $\zz[q]$.
%Define $\wr$ to be the set of irreducible complex characters of $W$ and $R_W$ its character ring over $\zz[q]$.

%\subsection{}  
Let $Z_G(e)$ be the centralizer of $e$ in $G$ under the adjoint action and let $A(e) := Z_G(e)/ Z^0_G(e)$ be the component group of $e$.
Then $H^*(\flag_e)$ carries a representation of $A(e)$ and the $A(e)$-action commutes with the $W$-action.

The cohomology of $\flag_e$ vanishes in odd degrees (see \cite{bs:green}, \cite{dlp}).  We define 
$Q_{e} \in R(W \times A(e))$ by
%$$Q_{e} = \sum_{(\chi, \phi) \in \wr \times \ar} \big( \sum_{j \geq 0} \langle H^{2j}(\flag_e), \chi \otimes \phi \rangle q^j \big) \, \chi \otimes \phi.$$
$$Q_{e} = \sum_{j \geq 0}  H^{2j}(\flag_e) q^j$$
and for $\phi \in \ar$ define $Q_{e, \phi} \in R(W)$ by 
$$Q_{e, \phi} =   \langle Q_e, \phi \rangle_{A(e)}.$$
Then
$$Q_e = \sum_{\phi \in \ar} Q_{e, \phi} \phi.$$ 

%It also makes sense to consider $Q_e$ in the Grothendieck group of $W$ with coefficients in $R(A(e))$.
%Doing so, we define the multi-set of pairs $$\{(m_1, \pi_1),  (m_2, \pi_2), \dots, (m_s, \pi_s)  \},$$ where $m_j$ is a nonnegative integer and $\pi_j \in \ar$, by 
Next, we define %the multi-set of pairs
$\{(m_1, \pi_1),  (m_2, \pi_2), \dots, (m_s, \pi_s)  \},$ where $m_j \in \zz_{\geq 0}$  and $\pi_j \in \ar$, by 
$$\langle Q_{e}, V \rangle_{W} = q^{m_1} \pi_1 + q^{m_2} \pi_2 + \cdots + q^{m_s} \pi_s,$$
an identity in $R(A(e))$.
%In other words, the $(m_j, \pi_j)$'s measure the occurrences of $V$ in the total Springer representation $H^*(\flag_e)$.

%where $\chi = \chi \otimes 1$ and $\phi = 1 \otimes \phi$ on the right-hand side.

%We call $Q_e$ a Green function as in \cite{bs:green}.   (okay?)

%Define $Q_{e, \phi} \in R_W$ as the coefficient of $\phi$ in the expansion:
%$$Q_e = \sum_{\phi \in \ar} Q_{e, \phi} \phi.$$

%For $\phi \in \ar$ let $Q_{e, \phi} \in R_W$ denote the Green function \cite{}.
%Explicitly, 
%$$Q_{e, \phi} = \sum_{\chi \in \wr} \langle H^{2j}(\flag_e), \chi \otimes \phi \rangle q^j \chi.$$
%We set 
%$$Q_e = \sum_{\phi \in \ar} Q_{e, \phi} \phi,$$ an element of $R_{W \times A(e)}$.

%denote the multi-set consisting of half the degrees of the occurrences of $V$ in $H^*(\flag_e)$.  In other words, 
%$$\sum_{j \geq 0} \langle H^{2j}(\flag_e), V  \rangle q^j  = q^{m_1} + q^{m_2} + \cdots + q^{m_s},$$  forgetting the $A(e)$-action
%and using the notation $V$ for both the representation and the character of the representation.
%In fact, we will need the $A(e)$-action so we record the multi-set of pairs $\{(m'_j, \pi_j)\}$ where $m'_j$ is a positive integer and $\pi_j \in \ar$ defined so that 
%the coefficient of $V$ in $Q_e$ is 

In the classical groups, the computation of the $m_j$'s 
was carried out by Lehrer-Shoji \cite{lehrer-shoji:reflections} in most cases and then completed by Spaltenstein \cite{spaltenstein:reflection}. 
The computation of the $m_j$'s in the exceptional groups is handled by studying tables of Green functions \cite{bs:green}.
%We are able to give a new proof of these computations
%in types $A_n, B_n, C_n$, but we need to rely on Spaltenstein's work in type $D_n$. 

The computation of the $\pi_j$'s seems to be new to the present paper.  
It turns out that at most one $\pi_j$ is nontrivial and 
this can only occur when $\g$ is not of type $A_n, B_n$ or $C_n$.   
We reserve $\pi_s$ for this possibly nontrivial representation.
The calculation of $\pi_s$ will be given in Section \ref{pis for D} for type $D_n$.   The calculation of
$\pi_s$ in the exceptional groups is again handled by studying tables in \cite{bs:green}.

When $e$ is regular in a Levi subalgebra, Lehrer and Shoji \cite{lehrer-shoji:reflections}
conjectured that the occurrences of $\wedge^i V$ in $H^*(\flag_e)$ are given by 
\begin{equation}\label{ls_conj}
\sum_{\phi \in \ar}  \dim \phi \sum_{i=0}^{n}  \langle Q_{e, \phi}, \wedge^i V \rangle \, y^i = \prod_{j=1}^{s} (1+yq^{m_j})
\end{equation}
%($\dim \phi = 1$ in this case as just noted).
The purpose of this paper is to establish the Lehrer-Shoji conjecture in all types and to extend it in two ways:  (1) to the case of all 
nilpotent elements $e$; and (2) to incorporate the $A(e)$-action on $H^*(\flag_e)$.   
Our main theorem is
  
\begin{theorem} \label{weaker_theorem}
Let $d = \dim(\pi_s)-1$ when $e$ is $F_4(a_3)$ or $E_8(a_7)$ and $d = \dim(\pi_s)$, otherwise.
Then the following identity holds in $R(A(e))[y]$
\begin{equation*} 
%\sum_{i=0}^{n} \langle Q_{e}, \wedge^i V \rangle  y^i =  (1 + yq^{m_s}\pi_s + y^2 q^{2m_s}\!\wedge^2 \! \pi_s + \cdots + y^d q^{dm_s} \! \wedge^d \! \pi_s) \prod_{j=1}^{s-1} (1+yq^{m_j}) 
\sum_{\phi \in \ar} \sum_{i=0}^{n} \langle Q_{e, \phi}, \wedge^i V \rangle y^i \, \phi=  (1 + yq^{m_s}\pi_s + y^2 q^{2m_s}\!\wedge^2 \! \pi_s + \cdots + y^d q^{dm_s} \! \wedge^d \! \pi_s) \prod_{j=1}^{s-1} (1+yq^{m_j}).
\end{equation*}
\end{theorem}

Aside from the two exceptional cases, Theorem \ref{weaker_theorem}
would be immediate if Conjecture \ref{conj1} were true since there is only
one nontrivial $\pi_j$.  
Regarding the two cases where $e$ is $F_4(a_3)$ or $E_8(a_7)$, 
the value $d = \dim(\pi_s)-1$ is the best possible 
since $Q_{e,\phi}=0$  for $\phi = \wedge^{\text{top}} \pi_s$.

On the other hand, we can use Theorem \ref{weaker_theorem}
to establish Conjecture \ref{conj1} in types $B_n$ and $C_n$ for all $e$ using Henderson's argument \cite{henderson:exterior}.

When $e$ is regular in a Levi subagebra,  it turns out that
$A(e)$ is always elementary abelian and therefore that $\pi_s$ is one-dimensional. 
Consequently by evaluating at the identity element of $A(e)$, 
Theorem \ref{weaker_theorem} reduces to Equation \ref{ls_conj}
and thus to the original conjecture of Lehrer and Shoji.
%This reduction amounts to forgetting about the $A(e)$-action on $H^*(\flag_e)$.

In the situation where $e$ is regular in a Levi subalgebra $\levi$, moreover, the $m_j$'s are exactly the numbers computed by Orlik and Solomon arising from the restricted hyperplane arrangement for $W$ defined by the Weyl group $W_{\levi}$ of $\levi$ (see \cite{lehrer-shoji:reflections}, \cite{spaltenstein:reflection}).   
%\marginpar{\tiny{include ref to OS, Saito?, \\or later}}
In the second half of the paper, we give an explanation for this coincidence.

%Finally we remark that $\pi_s$ is nontrivial exactly when $N_W(W_{\levi}) / W_{\levi}$ fails to be a reflection group in its natural representation, the cases of which were determined by Howlett \cite{howlett}.  

%%%re-check this!!!!

\section{\texorpdfstring{Shoji's recursive formula for $Q_e$}{Shoji's recursive formula for Qe}} \label{section:shoji}

\subsection{}

Shoji \cite{shoji:green_f4}, \cite{shoji:green_classical}  was able to give a recursive formula for $Q_e$ by using the orthogonality of Green functions 
in Springer's original work \cite{springer:green}.  Beynon-Spaltenstein \cite{bs:green} used Shoji's algorithm to compute $Q_e$ for 
the exceptional groups, but outside of type $A_n$, there are no known closed formulas in general for $Q_e$.   

To use Shoji's algorithm it is necessary to work over a finite field $\F_q$ since the algorithm involves the number of points of nilpotent orbits over 
$\F_q$.  Originally, there were restrictions arising from \cite{springer:green} that both $q$ and the characteristic $p$ of $\F_q$ needed to be large, 
but these were relaxed by Lusztig in his work on character sheaves \cite[Theorem 24.8]{lusztig:characterV} where only
$p$ good for $G$ is needed.  Lusztig also showed, although we will not need it, that the number of $\F_q$-points 
of the nilpotent orbits can actually be deduced from his algorithm; one only needs to know the Springer correspondence 
and the fake degrees of the tensor product of 
pairs of irreducible representations of $W$.

\subsection{}
Let $F: G \to G$ be a Frobenius morphism defining a split $\F_q$-structure on $G$, where $q=p^n$.  
%As long as $p$ is good, the nilpotent orbits of $G$ in $\g$ is independent of $p$ \cite{}.
Let $S = S^*(V)$ denote the symmetric algebra on $V$.  We will view $S \in R(W)$ with $V$ having degree $1$ in $q$.
We use the results from \cite{shoji:green_f4}, \cite{shoji:green_classical}, \cite{bs:green}.% \cite{bs:computer}.

First for $e=0$, $Q_e = Q_0$ is isomorphic to the coinvariants of $S$ and hence by Chevalley's result
$$S = \frac{Q_0}{\prod_{i=1}^n (1 - q^{d_i}) }$$
where $\{d_1, d_2, \dots, d_n\}$ are the fundamental degrees of $W$.
The Chevalley-Steinberg formula for $|G^F|$ is
$$|G^F| = q^N \prod_{i=1}^n (q^{d_i} -1)$$
where $N$ is the number of positive roots in a root system for $G$.  
Hence as an identity in $R(W)$,
\begin{equation} \label{harmonics}
q^{N} Q_0 =  (-1)^n |G^F| S.
\end{equation} 

Next, we use the orthogonality formula.  
For each nilpotent $G$-orbit $\orbit$ in $\g$, we select a representative $e \in \orbit^F$ which is {\it split} in the language of %\cite[]{bs:green}
\cite{bs:green}
(this is called ``distinguished'' in  \cite{shoji:green_classical}).  This is possible for every orbit in every $\g$ except for one orbit in $E_8$ \cite{shoji:green_classical}\cite{bs:green}
when $q \equiv -1 \mbox{ (mod } 3)$.  Nonetheless, since this orbit, denoted $E_8(b_6)$, is distinguished in the usual sense of nilpotent orbits, the number of points of any of its three rational orbits is a monomial in $q$, independent of  $q$, and the orthogonality formula will still hold since it holds for an infinite number of $q$.

From now on $e$ is assumed to be rational and split.  
The $G^F$-orbits on $\orbit^F$ are parametrized by the conjugacy classes in $A(e)$ and we denote by $e_c$ a representative indexed by the class $c \subset A(e)$.  Then define 
\begin{equation} \label{eqn:rational}
Q'_{e_c} := \sum_{\phi \in \ar}  \phi(c) Q_{e, \phi},
\end{equation} 
where we also use $c$ to denote a representative from this conjugacy class.

Let $\epsilon$ denote the sign character of $W$.
For any $\chi \in \wr$ the orthogonality formula as presented in \cite{shoji:green_f4} is 
\begin{equation}  \label{shoji:orthogonal}
q^N  Q_0 \otimes \chi \otimes \epsilon= \sum_{e} \sum_{c} | \orbit_{e_c} | \langle Q'_{e_c}, \chi \rangle Q'_{e_c},
\end{equation} 
where the outer sum is over our chosen split representatives of the nilpotent $G$-orbits, the inner sum is over the conjugacy classes in $A(e)$,
and $\orbit_{e_c}$ is the $G^F$-orbit through $e_c \in \g^F$.   Since $p$ is good, the summations are independent of $q$ and 
$|\orbit_{e_c}| = P(q)$  for some polynomial $P(x) \in \qq[x]$, independent of $q$.

Using Equation \ref{eqn:rational} twice, Equation \ref{shoji:orthogonal} can be rewritten as
\begin{equation} \label{shoji:orthogonal2}
q^N  Q_0 \otimes \chi \otimes \epsilon =  \sum_{e} \sum_{c}  \sum_{\phi' \in \ar}  \sum_{\phi \in \ar} |\orbit_{e_c}| \phi'(c) \langle Q_{e, \phi'}, \chi \rangle  \phi(c) Q_{e, \phi},
\end{equation} 

Next set $\chi = \wedge^{n-j} V$ in Equation \ref{shoji:orthogonal2}.  Then $\chi \otimes \epsilon = \wedge^{j} V$ since $\wedge^n V = \epsilon$ and $V \simeq V^*$.
Then combining with Equation \ref{harmonics} and changing the order of the summations gives

\begin{equation}  \label{equation:1}
(-1)^n  |G^F| S \otimes \wedge^j V  =  \sum_{e} \sum_{\phi} \left[ \sum_c  \sum_{\phi'} \langle Q_{e, \phi'}, \wedge^{n-j} V \rangle \phi'(c) |\orbit_{e_c}|  \phi(c) \right]  Q_{e, \phi},
\end{equation}
an identity in $R(W)$. % In fact, the coefficients of any $W$-representation on either side lies in $\zz[q]$.

\subsection{}\label{subsection:gns}

Our proof of Theorem \ref{weaker_theorem} in the classical groups will use results of Gyoja, Nishiyama, and Shimura \cite{gyoja:invariants1}
who studied the two-variable functions $\ttau(\chi) \in \qq(q)[y]$ for each $\chi \in \wr$   
defined by 
$$\ttau(\chi) :=  \sum_{i, j} \langle S^i(V) \otimes \wedge^j V, \chi \rangle q^i y^j.$$
They computed $\ttau(\chi)$ for each $\chi \in \wr$ and showed that the factorization pattern of $\ttau(\chi)$ groups the irreducible
representations of $W$ into packets that often coincide with the grouping 
arising from the  two-sided Kahzdan-Lusztig cells in $W$.   

In order to use their results we 
multiply Equation \ref{equation:1} by $y^j$, sum up over $j$, and take the inner product with $\chi \in \wr$, to get
an identity in $R(W)[y]$:
\begin{equation}  \label{transit_equation}
(-1)^n |G^F| \ttau(\chi) =   \sum_{e}  \sum_{\phi} \left[ \sum_c  \sum_{\phi'} \sum^n_{j=0} \langle Q_{e, \phi'}, \wedge^{n-j} V \rangle  \phi'(c) |\orbit_{e_c}|   \phi(c) y^j \right]  \langle Q_{e, \phi}, \chi \rangle 
\end{equation}
Introduce the notation for the term in brackets above:
\begin{equation} \label{h_equation}
h_{e, \phi} := \sum_c  \sum_{\phi'} \sum^n_{j=0} \langle Q_{e, \phi'}, \wedge^{n-j} V \rangle  |\orbit_{e_c}|  \phi'(c) \phi(c) y^j ,
\end{equation}
an element of $\qq[q,y]$ which depends on $e$ and $\phi \in \ar$.
Then Equation \ref{transit_equation} becomes
\begin{equation}  \label{Main_equation}
(-1)^n |G^F| \ttau(\chi) =   \sum_{e}  \sum_{\phi} h_{e, \phi} \langle Q_{e, \phi}, \chi \rangle 
\end{equation}

%$$|G| \tau(\chi) =   \sum_{e} \sum_{j=1}^n (-1)^n   \sum_{\phi \in \ar}  g_{e, \phi}  \langle Q_{e, \phi}, \wedge^{n-j} V \rangle \langle Q_{e, \phi}, \chi \rangle y^j$$ 
%$$|G| \tau(\chi) =   \sum_{e} [ \sum_{j=1}^n (-1)^n  \sum_{\phi \in \ar}  [\sum_c  \phi(c) | \orbit_{e_c}| ]  \langle Q_{e, \phi}, \wedge^{n-j} V \rangle \langle Q_{e, \phi}, \chi \rangle y^j$$
%\begin{equation}  \label{main_equation}
%(-1)^n |G| \tau(\chi) =   \sum_{e}  \sum_{\phi'} \sum_c  [   \sum_{j=1}^n   \sum_{\phi} \langle Q_{e, \phi}, \wedge^{n-j} V \rangle \phi(c) y^j ]  \phi'(c) |\orbit_{e_c}|  \langle Q_{e, \phi'}, \chi \rangle 
%\end{equation}

\section{Proof of Theorem \ref{weaker_theorem} }

%\subsection{$B_n$, $C_n$, $D_n$}

In Theorem \ref{weaker_theorem}  we would like to show that
$$\sum_{\phi \in \ar}  \sum^n_{i=0} \langle Q_{e, \phi}, \wedge^i V \rangle  y^i \phi=  
(1 + yq^{m_s}\pi_s + y^2 q^{2m_s}\!\wedge^2 \! \pi_s + \cdots + y^d q^{dm_s} \! \wedge^d \! \pi_s) \prod_{j=1}^{s-1} (1+yq^{m_j}). $$
Replacing $y$ by $y^{-1}$, multiplying through by $y^{n}$, and replacing $i$ by $n-i$, gives the equivalent identity
%\begin{eqnarray} 
\begin{equation} \label{theorem:variant}
\sum_{\phi \in \ar}  \sum^n_{i=0}   \langle Q_{e, \phi}, \wedge^{n-i} V \rangle  y^i \phi  =  
y^{n-d-s+1}  (y^d + y^{d-1} q^{m_s}\pi_s + %y^{d-2} q^{2m_s}\!\wedge^2 \! \pi_s + 
\cdots + q^{dm_s} \! \wedge^d \! \pi_s) \prod_{j=1}^{s-1} (y + q^{m_j}).
%\end{eqnarray}
\end{equation}
%We will prove this equivalent identity in the classical groups by calculating $h_{e, \phi}$.
%by induction on the dimension of the orbit through $e$ and using 

We use the work of Gyoja, Nishiyama, and Shimura \cite{gyoja:invariants1} and Spaltenstein \cite{spaltenstein:reflection} to prove this equivalent identity in the classical types.  The proof involves 
calculating $h_{e, \phi}$ by induction on the dimension of the orbit through $e$.   
We omit the type $A_n$ case, which could be handled in the same manner, since $e$ is always regular in a Levi subalgebra and $A(e)$ is trivial and so the result can already be found in \cite{henderson:exterior}.

\subsection{Notation on partitions}

Let $$\lambda = [\lambda_1 \geq \lambda_2 \geq \dots \geq \lambda_{k-1} \geq \lambda_k > 0]$$
be a partition of $m$.
We set $|\lambda| = m$, $l(\lambda) = k$, and 
$$n(\lambda) = \sum_{i=1}^k (i-1)\lambda_i$$
and let $\dualp$ denote the dual or tranposed partition of $\lambda$.

Viewing a partition $\lambda$ as a Young diagram in the usual way, let $x = (i,j)$ denote the coordinates of a box
in the diagram where the top, left box is $(0,0)$.  
In other words, $x = (i,j)$ is a valid pair of coordinates if and only if $0 \leq i \leq k-1$ and $0 \leq j \leq \lambda_{i+1}-1$.
We write $x \in \lambda$ whenever $x = (i,j)$ is a valid pair of coordinates for the partition $\lambda$.

For $x \in \lambda$, we denote the {\it hook length} of x by 
$$h(x) = \lambda_{i+1} - j + \dualp_{j+1} - i -1$$
and the {\it content} of x by
$$c(x) = j-i.$$

\subsection{\texorpdfstring{Representations of $W$}{Representations of W}}

Recall that the irreducible representations of $W$ in types $B_n, C_n,$ and $D_n$ are parametrized by pairs of partitions $(\alpha, \beta)$ such that $|\alpha| + |\beta| = n$.  The representations 
will be denoted as $\chi^{\alpha, \beta}$.  In $D_n$ when $\alpha \neq \beta$
the representations $\chi^{\alpha, \beta}$ and $\chi^{\beta, \alpha}$ are equivalent, whereas when $\alpha = \beta$
there are two inequivalent representations $\chi^{\alpha, \alpha}_I$ and $\chi^{\alpha, \alpha}_{II}$ .

%The proof of the following can be found in \cite{}. (look up in Macdonald section I.2 and I.3 as per gyoja)
Regarding the bipartition attached to the irreducible representation $\wedge^j V$, we have 

\begin{lemma}[\cite{macdonald:book}]
In types $B_n, C_n, D_n$, the irreducible representation 
$\wedge^j V$ of $W$ is parametrized by the pair of partitions $(\alpha, \beta)$ with $\alpha = [n-j]$ and $\beta = [1^j]$.
%$([n-j], [1^j])$.
\end{lemma}

\subsection{\texorpdfstring{Springer correspondence for $\wedge^j V$}{Springer correspondence for exterior powers of V}}

Next we look at the nilpotent orbit attached to $\wedge^j V$ under the Springer correspondence.
Recall that nilpotent orbits in type $B_n$ are parametrized by partitions of $2n+1$ where 
even parts occur with even multiplicity;
in type $D_n$, by partitions of $2n$ where 
even parts occur with even multiplicity;
and in type $C_n$, by partitions of $2n$ where 
odd parts occur with even multiplicity.

For $j = 0$, $\wedge^j V$ is the trivial representation and so is attached to the regular nilpotent orbit,
with partition $[2n+1], [2n], [2n-1, 1]$ in types $B_n, C_n, D_n$, respectively.
For $j =  n$, $\wedge^j V$ is the sign representation and so is attached to the zero nilpotent orbit,
with partition $[1^{2n+1}], [1^{2n}], [1^{2n}]$ in types $B_n, C_n, D_n$, respectively.
For the other cases, we have

\begin{lemma} \label{springer:wedgie}
Under the Springer correspondence, $\wedge^j V$ for $j \not\in \{ 0, n \}$
is attached to the nilpotent orbit $\orbit_j$ with partition:
\begin{itemize} 
\item Type $B_n$:  $[2n -2j + 1, 1^{2j}]$
\item Type $C_n$:  $[2n -2j, 2, 1^{2j -2}]$
\item Type $D_n$:  $[2^{2}, 1^{2n-4}]$ if $j=n-1$ and $[2n -2j-1, 3, 1^{2j -2}]$ otherwise
\end{itemize} 
In types $B_n$ and $C_n$, the component group $A(e)$ is isomorphic to $S_2$ for $e \in \orbit_j$
and the corresponding local system under the Springer correspondence is trivial.
In type $D_n$, $A(e)$ is trivial for $j \in \{1,n-1\}$, while $A(e) \simeq S_2$ for $2  \leq  j \leq n-2$ and the corresponding local system is nontrivial.
\end{lemma}

\begin{proof}
Apply the algorithms in \cite[Chapter 13.3]{carter:book}.   We omit the details.  %  Show details ?.
\end{proof}

\begin{corollary} \label{cor:springer:wedgie}
Let $e$ be a nilpotent orbit parametrized by the partition $\lambda$.  Let $k = l(\lambda)$ be the number of parts of $\lambda$.
Then $\wedge^j V$ does not occur in $H^*(\flag_e)$ whenever
\begin{itemize} 
\item $j > \frac{k-1}{2}$ in type $B_n$,
\item $j > \frac{k}{2}$ in types $C_n$ and $D_n$.
\end{itemize} 
\end{corollary}

\begin{proof}
Under the hypotheses, the partition for the 
 nilpotent orbit $\orbit_j$ attached to $\wedge^j V$ in Lemma \ref{springer:wedgie} has more parts than the partition for $e$.  Consequently,
 $e$ is not contained in the closure of $\orbit_j$.  It follows from \cite[Corollary 2]{borho-mac} (see also \cite[Theorem 4.4]{shoji:green_classical}) that the Springer representations attached to 
 $\orbit_j$ cannot occur in $H^*(\flag_e)$. 
\end{proof}

\begin{remark}
The inequality in the corollary turns out to be optimal in types $B_n$ and $C_n$, but it is not optimal in general in $D_n$.  For example, in $D_4$ with $e = [3,3,1,1]$, then $\langle Q_{e}, \wedge^2 V \rangle = q^3$ and $\langle Q_{e}, \wedge^3 V \rangle = 0$ and the result is optimal.  On the other hand, with $e = [3,2,2,1]$ then already $\langle Q_{e}, \wedge^2 V \rangle = 0$.   This failure makes the inductive proof that we give in type $D_n$ a bit more complicated. 
\end{remark}

\subsection{Springer correspondence, in general}

In the next lemma we give a modest constraint on the bipartitions $(\alpha, \beta)$ that can be attached to an arbitrary nilpotent element $e$ under
the Springer correspondence.

\begin{lemma}  \label{lem:constraint}
Let $e$ be a nilpotent element in type $B_n, C_n,$ or $D_n$ with partition $\lambda$.
Let $(\alpha, \beta)$ be the Springer representation attached to  $e$ and local system $\phi \in \ar$.  Let $k = l(\lambda)$. 
% be the number parts of $\lambda$.  
\begin{itemize} 
\item Type $C_n$:  if $k$ is odd, then $l(\alpha) = \frac{k+1}{2}$ and  if $k$ is even, then $l(\beta) = \frac{k}{2}$;
%\item Type $B_n$:  if $\phi(a_1) = 1$, then $\beta$ has exactly $\frac{d-1}{2}$ parts;
%if $\phi(a_1) = -1$, then $\alpha$ has exactly $\frac{d+1}{2}$ parts;
\item Type $B_n$:  either $l(\alpha) = \frac{k+1}{2}$ or $l(\beta) = \frac{k-1}{2}$  ($k$ is always odd);
\item Type $D_n$:  either $l(\alpha) = \frac{k}{2}$ or $l(\beta) = \frac{k}{2}$ ($k$ is always even).
\end{itemize} 
\end{lemma}

\begin{proof}
The proof uses the algorithms in \cite[Chapter 13.3]{carter:book}.   We sketch the arguments:

\medskip

\textbf{Type $C_n$}.   Starting with $\lambda$ we will calculate the bipartition $(\xi, \eta)$ corresponding to the trivial local system on $\orbit_e$.
First we make sure that the number of parts of $\lambda$ is even by tacking on a $0$ to $\lambda$ if necessary (however, we will not adjust  
the definition of $k$).
Next, list the even parts of $\lambda$ in nondecreasing order as $x_1 \leq x_2 \leq \dots \leq x_{m-1} \leq x_{m}$.  Note that $m$ is guaranteed to be even since the number of odd parts is even and the total number of parts of $\lambda$ is now even by setting $x_1 = 0$ if necessary.

Chasing through the algorithm in \cite{carter:book}, we see that $x_i$ contributes 
$\frac{x_i}{2}$ to the bipartition: to $\xi$ if $i$ is even and to $\eta$ if $i$ is odd.   This is a result of 
the fact that the odd parts in $\lambda$ occur with even multiplicity.
%Next, consider the set of subscripts $I:= \{ i \ | \ \lambda_i \text{ is even } \}$ and order 

List the odd parts of $\lambda$ in nondecreasing order as $y_1 = y_2 \leq y_3 = y_4 \leq \dots \leq y_{2l-1} = y_{2l}$.
Let $y_{2i-1} = y_{2i} = y$ be consecutive numbers occurring in this sequence.
If the number of even parts of $\lambda$ less than $y$ is odd, then
the contribution of this pair is $\frac{y+1}{2}$ to $\xi$ and $\frac{y-1}{2}$ to $\eta$.
On the other hand, if the number of parts of $\lambda$ less than $y$ is even, then
the contribution of this pair is $\frac{y-1}{2}$ to $\xi$ and $\frac{y+1}{2}$ to $\eta$.

%Now we apply this description to calculate the beginning of the symbol attached to $e$.
Now assume $k$ is even.  If the smallest part of $\lambda$, which we have denoted $x_1$, is even, then $x_1$ is nonzero and $\eta$ begins will the nonzero number $\frac{x_1}{2}$.  Moreover each subsequent consecutive pair of parts of $\lambda$ contributes a number to $\xi$ and a number to $\eta$.  Since the contributions in this way make $\xi$ and $\eta$ nondecreasing, we see that $\eta$ contains exactly $\frac{k}{2}$ (nonzero) parts.
It follows that the symbol attached to $e$ begins  %see page ...
$$\begin{pmatrix}
     0  & & \dots \\
       & \frac{x_1}{2} + 1 &&\dots 
     \end{pmatrix}$$
and the second row of the symbol consists of exactly $\frac{k}{2}$ entries.
Moreover, the smallest nonzero number in the symbol is $\frac{x_1}{2} + 1$ and therefore any permutation of the entries of the symbol to produce an similar symbol must have an initial entry in the second row of at least $\frac{x_1}{2}+1$.  
We conclude that for any local system $\phi$ on $\orbit_e$ the corresponding bipartition $(\alpha, \beta)$ must have a $\beta$ with $\frac{k}{2}$ parts.

On the other hand, if the smallest part of $\lambda$ is odd, say $y$, then $\eta$ begins with 
$\frac{y+1}{2}$, a nonzero number, and contains $\frac{k}{2}$ (nonzero) parts.
It follows that the symbol attached to $e$ begins 
$$\begin{pmatrix}
     0  && \dots \\
       & \frac{y+1}{2} + 1&& \dots 
     \end{pmatrix}$$
and the second row of the symbol contains exactly $\frac{k}{2}$ entries.
Hence the smallest nonzero number in the symbol is $\frac{y+1}{2} + 1$ and as before 
we conclude that for any local system the corresponding bipartition $(\alpha, \beta)$ must have a $\beta$ with $\frac{k}{2}$ parts.

The situation is similar if $k$ is odd.   
%If the smallest part of $\lambda$ is even, say $x$, then $x_1 = 0$  and $x_2 =x$ and we see that $\alpha$ begins will the nonzero number $\frac{x}{2}$.   Likewise, if the smallest part of $\lambda$ is odd, say $y$, then $\beta$ begins with 
%$\frac{y+1}{2}$ , also a nonzero number.
Here, we deduce that the first row of the symbol begins with a nonzero number and contains exactly $\frac{k+1}{2}$ entries.
Thus any similar symbol must have a first row that begins with a nonzero number and  it follows that 
for any local system the corresponding bipartition $(\alpha, \beta)$ must have an $\alpha$ with $\frac{k+1}{2}$ parts.

\medskip

\textbf{Type $B_n$}.    Carrying out a similar analysis as in type $C_n$, we find that the symbol attached to $e$ has
$\frac{k+1}{2}$ entries in its first row and $\frac{k-1}{2}$ in its second row.  Moreover the initial entry in the second row is nonzero.
Hence any similar symbol provided by a nontrivial $\phi$ will produce a nonzero initial entry either in the first row or in the second row.
The precise condition that controls which of the two situations occurs depends on the value of $\phi(a_1)$ (see  \cite[Chapter 13.3]{carter:book} for notation).
If $\phi(a_1)=1$, then a nonzero entry will definitely occur in the second row
and if $\phi(a_1)=-1$, then a nonzero entry will definitely occur in the first row of the symbol.
It follows that for any local system the corresponding bipartition $(\alpha, \beta)$ must have either an $\alpha$ with $\frac{k+1}{2}$  parts or
a $\beta$ with $\frac{k-1}{2}$  parts.

\medskip

 \textbf{Type $D_n$}.   In this case, the symbol attached to $e$ has
$\frac{k}{2}$ entries in each row and the initial entry in the second row is nonzero.  In fact, 
the rows of the symbol are considered unordered in type $D_n$.
Hence any similar symbol provided by a nontrivial $\phi$ will produce a nonzero initial entry in one of the rows
and therefore for any local system, the corresponding bipartition $(\alpha, \beta)$ must have either an $\alpha$ or $\beta$ with $\frac{k}{2}$ parts.   As noted previously, the bipartition $(\alpha, \beta)$ is unordered.
\end{proof}

\subsection{\texorpdfstring{Explicit formula for $\ttau(\chi)$}{Explicit formula for chi}}

The functions $\ttau(\chi)$ from Section \ref{subsection:gns} were computed in \cite{gyoja:invariants1} using results from \cite{macdonald:book}.

In types $B_n$ and $C_n$,

\begin{equation}  \label{GNS_BC}
\ttau(\rep) =   q^{2n(\alpha) + 2n(\beta) + |\beta|} 
\prod_{x' \in \alpha}  \frac{1 + yq^{2c(x')+1}}{1 - q^{2h(x')}}
\prod_{x'' \in \beta}  \frac{1 + yq^{2c(x'')-1}}{1 - q^{2h(x'')}}
\end{equation}

In type $D_n$,

\begin{eqnarray} \label{GNS_D}
\ttau(\rep) =   \frac{q^{2n(\alpha) + 2n(\beta)}} {2 \prod_{x \in \alpha \cup \beta} (1 - q^{2h(x)} )} 
\lbrace
q^{|\beta|} \prod_{x' \in \alpha}  (1 + yq^{2c(x')+1})
\prod_{x'' \in \beta}  (1 + yq^{2c(x'')-1}) + \\ \nonumber
q^{|\alpha|} \prod_{x' \in \alpha}  (1 + yq^{2c(x')-1})
\prod_{x'' \in \beta}  (1 + yq^{2c(x'')+1}) 
\rbrace
\end{eqnarray}

and

\begin{eqnarray}
\ttau(\chi^{\alpha,\alpha}_I) =  
\ttau(\chi^{\alpha,\alpha}_{II}) =  
 q^{4n(\alpha) +|\alpha|}
\prod_{x \in \alpha}
\frac{
(1 + yq^{2c(x)+1}) (1 + yq^{2c(x)-1})
}
{
(1 - q^{2h(x)})^2
}
\end{eqnarray}

\medskip

The next lemma follows from Lemma \ref{lem:constraint} and the above formulas for $\ttau(\rep)$.

\begin{lemma}  \label{cor:divisors}
Let $\chi_{e, \phi}$ be the (nonzero) 
Springer representation attached to the nilpotent element $e$ with local system $\phi \in \ar$. 
Let $\lambda$ be the partition of $e$ and let $k = l(\lambda)$. % be the number of parts of $\lambda$.
Then 
%the numerator of 
$\ttau(\chi_{e, \phi})$ is divisible by 
$$1+q^{-c}y$$ 
in $\qq(q)[y]$ for 
\begin{itemize} 
\item Type $B_n$: $c \in \{ 1, 3, \dots, k-2 \}$
\item Type $D_n$: $c \in \{1, 3, \dots, k-3 \}$ 
\item Type $C_n$: $c$ odd with $c  \leq k-1$  ($k$ may be odd or even)
\end{itemize} 
\end{lemma}

\begin{proof}  In type $B_n$, 
if $\alpha$ has $\frac{k+1}{2}$ parts, then the values of $2c(x') +1$
as $x'$ runs down the first column of the Young diagram of $\alpha$ are $1, -1, -3, \dots,  2 -k$
and the result follows from Formula \ref{GNS_BC}.  %2(0 - (\frac{d-1}{2} - 1)) -1 =
The other possibility is that $\beta$ has $\frac{k-1}{2}$ parts, in which case the values of $2c(x'') -1$
as $x''$ runs down the first column of the Young diagram of $\beta$ are $-1, -3, \dots,  2 -k$
and the result follows from Formula \ref{GNS_BC}.  %2(0 - (\frac{d-1}{2} - 1)) -1 =

In type $C_n$ if $k$ is even, then $\beta$ has $\frac{k}{2}$ parts according the lemma.  
The values of $2c(x'') -1$
as $x''$ runs down the first column of the Young diagram of $\beta$ are $-1, -3, \dots,  1 -k$
and the result follows from Formula \ref{GNS_BC}.
On the other hand if $k$ is odd, then $\alpha$ has $\frac{k+1}{2}$ parts.
The values of $2c(x') +1$
as $x'$ runs along the first column of the Young diagram of $\alpha$ 
are $1, -1, -3, \dots,  2 -k$
and the result follows from Formula \ref{GNS_BC}.

In type $D_n$, we can assume that $\beta$ has $\frac{k}{2}$ parts according the lemma (the bipartition is unordered).
The values of $2c(x'') -1$ 
as $x''$ runs down the first column of the Young diagram of $\beta$ are $-1, -3, \dots,  1 -k$.  However,
the values of $2c(x'') +1$ 
as $x''$ runs down the first column of the Young diagram of $\beta$ are $1, -1, -3, \dots,  3 -k$,
which only guarantees by Formula \ref{GNS_D}
that $1+q^{-c}y$ in the stated range divides $\ttau(\chi_{e, \phi})$.
\end{proof}

\subsection{\texorpdfstring{Proof of Theorem  \ref{weaker_theorem} in types $B_n$ and $C_n$}{Proof in types Bn or Cn}}

Let $e$ be a nilpotent with partition $\lambda$ having $k$ parts.  We will prove 
$$ \sum_{\phi \in \ar} \sum_{j=0}^n \langle Q_{e, \phi}, \wedge^{n-j} V \rangle  y^j \phi =  y^{n-s} \prod_{i=1}^s (q^{2i-1}+ y)$$
where $s = \frac{k-1}{2}$ in type $B_n$
and $s = \lfloor  \frac{k}{2}   \rfloor$ in type $C_n$.   In particular the terms in the sum on the left are zero for $\phi$ nontrivial.
By looking at the coefficient of $y^{n-1}$
on both sides, which measures the occurrences of $V$ in $Q_{e, \phi}$, we will have proved that 
the $m_j$'s are $1, 2, \dots, 2s-1$ and all the $\pi_j$'s are trivial.  Of course the former was already known by  \cite{lehrer-shoji:reflections}, \cite{spaltenstein:reflection}.
Thus Equation \ref{theorem:variant} will hold and so will Theorem \ref{weaker_theorem}.

Define 
$$g_{e, \phi} = \sum_{j} \langle Q_{e, \phi}, \wedge^{n-j} V \rangle  y^j.$$  
We need to show that 
$g_{e, 1} = y^{n-s} \prod_{i=1}^s (q^{2i-1}+ y)$ and $g_{e, \phi} = 0$ for $\phi$ nontrivial.
Following \cite{shoji:green_classical},
let 
$$s_{\phi, \phi'} := \sum_{c} |\orbit_{e_c}| \phi(c) \phi'(c),$$
which is an element in $\qq[q]$.
Then Equation \ref{h_equation} becomes
$$h_{e, \phi} = \sum_{\phi'} s_{\phi, \phi'} g_{e, \phi'}.$$
From now on we will insist that $\phi, \phi' \in \ar$ actually appear in the Springer correspondence.
By a result of Shoji \cite{shoji:green_classical} $Q_{e,\phi}$ is nonzero if and only if 
$\phi$ occurs in the Springer correspondence.

%%First assume the result is true for smaller rank.   Don't need!
We proceed by induction on the dimension of $\orbit_e$.   
For the zero orbit, the result is known by Solomon \cite{solomon}.
Let $\chi_{e, \phi}$ be a (nonzero) Springer representation for $e$, with local system $\phi \in \ar$.  
Plug $\chi = \chi_{e, \phi}$ into  Equation \ref{Main_equation}
and use \cite{borho-mac} to obtain
\begin{equation}  \label{Main_equation_2}
(-1)^n |G^F| \ttau(\chi) =   h_{e,\phi}  q^{d(e)} + \sum_{e'} \sum_{\phi'} h_{e', \phi'} \langle Q_{e', \phi'}, \chi \rangle 
\end{equation}
where $d(e)$ is the dimension of $\flag_e$
and the sum on the right is now over orbit representatives $e'$ with $\orbit_{e'}  \subsetneq \overline{\orbit}_e$ and so by induction the result is known for $e'$.   
Now the partition $\lambda_{e'}$ for $e'$ is obtained from $\lambda$ by moving boxes down 
in the Young diagram for $\lambda$; in particular,
$$l(\lambda_{e'}) \geq l(\lambda)$$ and therefore
$$\prod_{i=1}^s (q^{2i-1}+ y)$$ divides $g_{e',1}$ in $\zz[q,y]$
and $g_{e', \phi'} = 0$ for nontrivial $\phi' \in \hat{A}(e')$.
Thus it divides all $h_{e', \phi'}$ in $\qq[q,y]$.
By Lemma \ref{cor:divisors} this product also divides the left side of Equation \ref{Main_equation_2}.
Hence it divides $h_{e, \phi}$ for all $\phi \in \ar$.

Now by Corollary \ref{cor:springer:wedgie} we know that 
$\wedge^{j} V$ does not occur in $Q_{e, \phi}$ if $j > s$ for any $\phi \in \ar$.
This translates into the fact that $y^{n-s}$ divides $g_{e, \phi}$ for each $\phi \in \ar$.
Since the degree of $h_{e, \phi}$ in $y$ is at most $n$, we conclude that for each $\phi \in \ar$ that 
$$h_{e, \phi} =  a_{\phi} y^{n-s} \prod_{i=1}^s (q^{2i-1}+ y),$$ 
where $a_{\phi} \in \qq[q]$. % \footnote{In fact, it lies in $\zz[q]$, say, by induction}.
In other words, 
$$\sum_{\phi'} s_{\phi, \phi'} g_{e,\phi'} = a_{\phi} y^{n-s} \prod_{i=1}^s (q^{2i-1}+ y).$$ 

Finally by \cite[Lemma 4.11]{shoji:green_classical}, the matrix 
$(s_{\phi, \phi'})$ is invertible over $\mathbf{Q}(q)$.
It follows that each $g_{e, \phi}$ is equal to $y^{n-s} \prod_{i=1}^s (q^{2i-1}+ y)$ times a polynomial in $\zz[q]$
since $g_{e, \phi}$ is of degree at most $n$ in $y$.
Examining the coefficient of $y^n$ in $g_{e, \phi}$, which keeps track of the trivial representation of $W$ in $Q_{e,\phi}$,
we see that $g_{e, \phi} = 0$ for $\phi$ nontrivial since the trivial representation does not occur
in $Q_{e,\phi}$ for $\phi$ nontrivial.  Moreover, the trivial representation only occurs in degree zero in $Q_{e,1}$ and it only occurs once,
so that 
$$g_{e, 1} = y^{n-s} \prod_{i=1}^s (q^{2i-1}+ y).$$
This completes the proof in types $B_n$ and $C_n$.  

\subsection{\texorpdfstring{Type $D_n$}{Type Dn}}

\subsubsection{Proof of Theorem \ref{weaker_theorem}.}

The proof proceeds as in types $B_n$ and $C_n$, by induction on the dimension of $\orbit_e$, except the induction only gets us so far in general.   

%First, by Theorem \ref{spalt:D}, $e$ always has $\{1, 3, \dots, k-3\}$ among the values of its $m_j$'s.  Since 
%the partition $\lambda_{e'}$ for any element $e'$ in the closure of $\orbit_e$ has $l(\lambda_{e'})  \geq k$, $e'$ also has 
%$\{1, 3, \dots, k-3\}$ among the values of its $m_j$'s.  This allows the induction to have some force.

Let $e$ be a nilpotent element with partition $\lambda$ having $k$ parts ($k$ must be even).
Imitating the $B_n/C_n$ proof, Lemma \ref{cor:divisors} and induction imply that
$$\prod_{i=1}^{\frac{k}{2} - 1} (q^{2i-1}+ y)$$ divides $h_{e, \phi}$ for all $\phi \in \ar$. 
%\footnote{Again we need only consider those $\phi$ appearing in the Springer correspondence} 
Next, by Corollary \ref{cor:springer:wedgie} we know that 
$\wedge^{j} V$ does not occur in $Q_{e, \phi}$ if $j > \frac{k}{2}$ for any $\phi \in \ar$.
Together this means that
$$b(q,y) := y^{n - \frac{k}{2}}\prod_{i=1}^{\frac{k}{2} - 1} (q^{2i-1}+ y)$$ divides $h_{e, \phi}$ for all $\phi \in \ar$.
Thus it divides $g_{e, \phi}$ for all $\phi \in \ar$ as in the type $B_n/C_n$ proof.

Now $b$ is of degree $n-1$ in $y$ and $g_{e, 1}$ is monic of degree $n$ in $y$ 
since the trivial representation only occurs in degree zero in $Q_{e,1}$.
Hence 
$$g_{e, 1} = b(q,y)(y + c_1)$$  with $c_1 \in \zz_{\geq 0}[q]$. 
On the other hand, for nontrivial $\phi$, $g_{e, \phi}$ is at most degree $n-1$ in $y$ since the trivial representation
does not occur in $Q_{e, \phi}$.  Hence in that case,
$$g_{e, \phi} = b(q,y)c_{\phi}$$  with $c_{\phi} \in \zz_{\geq 0}[q]$. 

Next we study the coefficient of $y^{n-1}$ in $g_{e,\phi}$, which measures the occurrences of $V$ in 
$Q_{e, \phi}$, 
i.e.\! the $m_j$'s.
By looking at $g_{e, 1}$, we conclude $1, 3, \dots, k-3$ occur among the $m_j$'s and these occur with $\pi_j$ trivial.
It turns out there is at most one additional $m_j$ (see Theorem \ref{spalt:D} below) and this is enough to finish the proof.  
%It would enough to finish the proof if we knew that there were at most one additional $m_j$.  Suppose this is the case.
First, if there is no additional $m_j$, then $g_{e,1} = b(q,y)y$ and $g_{e,\phi}=0$ for $\phi$ nontrivial and the proof is complete.
Second, if there is one additional $m_j$ with $\pi_j$ trivial, then 
$g_{e,1} = b(q,y)(y+q^{m_j})$ and $g_{e,\phi}=0$ for $\phi$ nontrivial and the proof is complete.
Finally, if there is an additional $m_j$ with $\pi_j$ nontrivial, then 
$$g_{e,\pi_j} = b(q,y)q^{m_j}$$ and 
$$g_{e,1} = b(q,y)y$$ and all other $g_{e, \phi}$ are zero and this is statement of the theorem in that case, completing the proof.
In Theorem \ref{spalt:D} and Proposition \ref{me:pij}, we give the explicit value of the $m_j$ and $\pi_j$ that the inductive proof misses.

%Next if all the $\pi_j$ are trivial, then $g_{e,\phi} = 0$ for $\phi$ nontrivial since otherwise it would be of degree exactly $n-1$ in $y$ 
%and therefore contribute a copy of $V$ to the $\phi$-isotypic component of $H^*(\flag_e)$, forcing some $\pi_j$ to equal $\phi$, a contradiction.
%We can complete the proof by 
%Theorem \ref{spalt:D} and Proposition \ref{me:pij}.
%Assume first that all the $\pi_j$ are trivial.  Then $g_{e,\phi} = 0$ for $\phi$ nontrivial since otherwise it would be of degree exactly $n-1$ in $y$ 
%and therefore contribute a copy of 
%$V$ to the $\phi$-isotypic component of $H^*(\flag_e)$, which would contradict Proposition \ref{me:pij}.

%In this case, $g_{e, 1}$ either equals $b(q,y)$ if $r$ is odd or $b(q,y)(y+q^{m_s})$ if $r$ is even  
%by looking at the coefficient of $y^{n-1}$ and 
%using Theorem \ref{spalt:D}.  This is exactly the statement of Theorem \ref{weaker_theorem}.
%
%On the other hand, if $\pi_s$ is nontrivial, then by looking at the coefficient of $y^{n-1}$,
%$$g_{e,\pi_s} = b(q,y)q^{m_s}$$ and 
%$$g_{e,1} = b(q,y)y$$ and all other $g_{e, \phi}$ are zero.  This completes the proof.

\subsubsection{\texorpdfstring{The values of $m_j$}{The values of fj}}

%Let $e$ be a nilpotent element with partition $\lambda$ having $k$ parts ($k$ must be even).
The values of the $m_j$'s were computed by Spaltenstein \cite{spaltenstein:reflection} extending the work in \cite{lehrer-shoji:reflections}.
%In particular the one possible value of $m_j$ that was missed in the previous proof is determined.

\begin{theorem}\label{spalt:D}   \cite{spaltenstein:reflection} 
Let $r$ be the number of parts of $\lambda$ which are not equal to $1$.  The number of occurrences of $V$ in $H^*(\flag_e)$ is 
$s =  \frac{k}{2}-1$ if $r$ is odd and $s = \frac{k}{2}$ if $r$ is even.

If $r$ is odd,
%$s = \frac{k-2}{2}$ and 
$$m_j = 2j-1  \text{  for }  j \in \{1, 2, \dots, s\}.$$

If $r$ is even,
%$s = \frac{k}{2}$ and 
$$m_j = 2j-1 \text{  for } j \in \{1, 2, \dots, s-1\}$$
and 
$$m_s = {\frac{k + r - 2}{2}}.$$
\end{theorem}
% \begin{theorem}   \cite{spaltenstein:reflection} 
%If $r$ is odd,
%$$\langle Q_{e}, V \rangle_W =  \sum_{j=1}^{s} q^{2j-1},$$
%where $s = \frac{d-2}{2}$.
%
%If $r$ is even,
%$$\langle Q_{e}, V \rangle =  q^{\frac{d + r - 2}{2}} + \sum_{j=1}^{s-1} q^{2j-1} ,$$
%where $s = \frac{d}{2}$.
%\end{theorem}

In other words, the missing value from the inductive proof is $m_s = {\frac{k + r - 2}{2}}.$

\subsubsection{\texorpdfstring{Computation of $\pi_j$}{Computations of pij}}  \label{pis for D}
The determination of the $\pi_j$'s was not given in  \cite{spaltenstein:reflection}, but can be deduced from the work there.

If $r$ is even and $k \neq r$, then $\lambda$ contains $k-r$ parts equal to $1$ and 
$e$ lies in a proper Levi subalgebra $\levi$ of $\g$ of type $D_{l}$ with $l = n - \frac{k-r}{2}$.   
If $\lambda$ contains an odd part different from $1$, then $A(e)$ will be nontrivial and moreover 
the component group of $e$ in $\levi$ defines an index two subgroup of $A(e)$, which we denote $H$.

\begin{proposition}   \label{me:pij}
For $j \in \{ 1, 2, \dots, s-1 \}$, $\pi_j$ is trivial.

If $r$ is odd, then $\pi_s$ is trivial.

If $r$ is even, then $\pi_s$ is trivial unless $k \neq r$ and $\lambda$ has an odd part different from $1$.
In that case, $\pi_s$ is the 
nontrivial representation of $A(e)$ which takes value $1$ on $H$.
\end{proposition}

\begin{proof}
Following the notation of and using the results from \cite{spaltenstein:reflection} (with $e$ in place of $A$ and $k$ in place
of $d$), we have a map
$$\pi^*_e:  H^*({\mathcal P}_e) \to H^*(\flag_e)^{W(P)}.$$
which is an isomorphism and which commutes with the action of $A(e)$ on both sides.
The $W(P)$-invariants are nonzero on only three representations of $W$:  the trivial representation, 
the reflection representation $V$, and a representation of dimension $n-1$ denoted by $\xi$.  Moreover the invariants
are one-dimensional in all of those cases.

There is also a map $$i^*_e: H^*({\mathcal P}) \to H^*({\mathcal P}_e),$$
which is surjective when $r$ is odd, $k+r=2n$, or $k=r$.  Note
that $k+r=2n$ if and only if $\lambda = [2^a, 1^b]$;  in particular, there are no odd parts different from $1$.
Since the image of $i^*_e$ lies in the $A(e)$-invariants of $H^*({\mathcal P}_e)$,  this proves the proposition in those cases.
In the remaining cases, $i^*_e$ has a one-dimensional cokernel in degree $m_s$, which implies that $\pi_j$ is trivial unless $j=s$.
Moreover, this cokernel corresponds to $V$ (and not $\xi$).
Hence, if we show that $A(e)$ acts nontrivially on the two-dimensional space $H^{2m_s}({\mathcal P}_e),$ then $V$ corresponds to a nontrivial character of $A(e)$.

%Therefore the proof is reduced to the case where $r$ is even and $k \neq r$.  
Now let $U = \ker(e)$ in the natural action of $e$ on $\mathbf{k}^{2n}$.   Then $U$ is of dimension $k$ and the bilinear form $\beta$ on $U$ has a radical $U_0$ of dimension $r$.  Let $U' \subset U$ be a subspace of type $(\frac{k+r}{2}+1, \frac{k+r}{2}-1)$.  Let $M \subset U$ be of dimension $k-r$ with $U = U_0 + M$ and with $\beta$ non-degenerate on $M$.  Let $O(M)$ denote the orthogonal group defined by the restriction of $\beta$ to $M$.  Note that since $k>r$ and $r$ is even (hence $k-r$ is even), $M$ is nonzero and even-dimensional. Then $U'$ contains exactly two subspaces of type $(\frac{k+r}{2}, \frac{k+r}{2})$ and these subspaces are interchanged by any determinant $-1$ element of $O(M)$  
and fixed by $SO(M)$.  It follows that the component group of $O(M)$ acts nontrivially on the two irreducible components of ${\mathcal Q}_{U'}$.
Thus it acts nontrivially on $H^{2m_s}({\mathcal Q}_U')$ and hence on $H^{2m_s}({\mathcal Q}_U) \simeq H^{2m_s}({\mathcal P}_e)$.

To complete the proof we note that when $r$ is even and $k \neq r$, then $A(e)$ is nontrivial if and only if $\lambda$ possesses an odd part bigger than $1$ (in which case it must have an even number of odd parts bigger than $1$).   
The image of $O(M)$ lies in $Z_G(e)$ and its image in $A(e)$, which is of order two, has trivial intersection with $H$.  
It follows that the character $\pi_s$ in question is the unique one which is trivial on $H$.
\end{proof}

\subsection{Exceptional groups}  \label{exceptionals}

We verified the theorem by looking at the tables in \cite{bs:green}.  %\cite{bs:computer}.  
%The verification is simplified slightly by making use of the same induction as in type $D_n$, which gives much (but not usually all) of $g_{e,1}$.

\section{Establishing Conjecture \ref{conj1} in $B_n$ and $C_n$}

Henderson \cite{henderson:exterior} showed that Conjecture \ref{conj1} is true when $e$ is regular in a Levi subalgebra in types $A_n, B_n,$ and $C_n$.  
The fact that Theorem \ref{weaker_theorem} is true for all $e$ in types $B_n$ and $C_n$ is sufficient
to show that the conjecture holds using Henderson's argument (see Theorem 1.2 in \cite{henderson:exterior}).

\begin{theorem}
For any nilpotent element $e$ in type $B_n$ or $C_n$, 
the algebra $$\displaystyle ( \bigoplus_{i=0}^n H^*(\flag_e) \otimes \wedge^i V)^W$$ is an exterior algebra on the subspace $(H^*(\flag_e) \otimes V)^W.$
\end{theorem}

\begin{proof}
Recall that $e$ corresponds to a partition $\lambda$ of length $l(\lambda)$.  The values of 
$m_1, m_2, \dots, m_s$ for $e$  are 
$$1, 3, \dots, 2s-1$$
where $s = \frac{l(\lambda)-1}{2}$ in type $B_n$
and $s = \lfloor  \frac{l(\lambda)}{2}   \rfloor$ in type $C_n$.
%This was established in the course of the proof of Theorem \ref{weaker_theorem} and was already known from \cite{lehrer-shoji:reflections}, \cite{spaltenstein:reflection}. 

As in \cite{henderson:exterior}
consider the natural homomorphism
$$\psi: \wedge^*(( H^*(\flag_e) \otimes V)^W)  \to (H^*(\flag_e) \otimes \wedge^* V)^W.$$
Since Theorem \ref{weaker_theorem} holds, the dimensions of both the domain and range of $\psi$ are equal to $2^s$. Thus it is sufficient to show that $\psi$ is injective.

Consider a Levi subalgebra $\levi_s$ of type $B_s$ (respectively, $C_s$) in type $B_n$ (respectively, $C_n$).    The usual exponents of $\levi_s$ are $1, 3, \dots, 2s-1$ and so these coincide with the numbers $m_1, m_2, \dots, m_s$ for $e$.  Now Henderson already noted that 
there exist fundamental invariants of $W$ on $V$ that restrict to a set of
fundamental invariants for the Weyl group of $\levi_s$, acting on its
reflection representation as a subrepresentation of $V$.
Hence condition (1) of Theorem 1.2 in 
 \cite{henderson:exterior} holds.  It remains to show that condition (2) of that Theorem holds as well.
 
Let $\mathfrak{p}_s$ be a parabolic subalgebra with $\levi_s$.  If we can show that nilpotent orbit of $e$ intersects the nilradial of $\mathfrak{p}_s$, then condition (2) holds and Theorem 1.2 in \cite{henderson:exterior} applies to show that 
$\psi$ is injective, whence it is an isomorphism. 

The Richardson orbit of $\mathfrak{p}_s$ has partition $\mu = [2n-2s+1, 1^{2s}]$ in $B_n$ 
and has partition $\mu = [2n-2s, 2, 1^{2s-2}]$ in $C_n$.  
%\marginpar{\tiny{\hspace{1in} cite?}}
In type $B_n$, $\lambda$ has length $2s+1$ and so $\mu$ dominates $\lambda$.  In type $C_n$ , $\lambda$ has length either $2s$ or $2s+1$.  Since odd parts of $\lambda$ occur in pairs, $\mu$ also dominates $\lambda$.  Consequently $e$ lies in the closure of the Richardson orbit and therefore the orbit through $e$ intersects the nilradical of $\mathfrak{p}_s$ since the $G$-saturation of the nilradical coincides with the closure of the Richardson orbit of the nilradical.  Therefore condition (2) holds.
\end{proof}

\begin{remark}
For most other cases the conjecture is open.
We can verify the conjecture for the minimal orbit in all types.
In type $D_n$ we observe that the proof of Theorem \ref{spalt:D} implies that
all copies of $V$ in $H^*(\flag_e)$ (except the one in degree $m_s$ when $r$ is even, $k+r \neq 2n$ and $k \neq r$)
are images of copies of $V$ in $H^*(\flag)$.  But the latter does not seem to imply anything about 
the images of the copies of  $\wedge^j V$ when $j >1$.
\end{remark}

\section{\texorpdfstring{A decomposition of the $t^n$-representation}{A decomposition of the tn-representation}}

When $e$ is regular in a Levi subalgebra, 
the $m_j$'s are coincident with the Orlik-Solomon exponents coming from the theory of hyperplane arrangements.  The aim of this section is to explain this coincidence.

\subsection{The restricted hyperplane arrangements $\mathcal{A}^J$}

Let $\Pi \subset \ro$ be a set of simple roots contained in a set of roots for $W$.  
For $J \subset \Pi$ let $W_J \subset W$ be the corresponding parabolic subgroup.
There is a hyperplane arrangement $\mathcal{A}^J$ in $V^{W_J}$ obtained by intersecting all the root hyperplanes for $W$ (which do not contain $V^{W_J}$) with $V^{W_J}$.  The characteristic polynomial of $\mathcal{A}^J$ is denoted $\chi_J(x)$. These characteristic polynomials were computed by Orlik and Solomon \cite{os:exponents}.    They showed that these polynomials factor into linear factors with positive integer roots:
$$\chi_J(x) = \prod_{j=1}^{n  - |J|} (x - e_j).$$
As mentioned in the introduction, these roots coincide with the $m_j$'s in the first part of the paper when $e$ is regular in the Levi subalgebra corresponding to $J$,
something we will give an explanation for in Corollary \ref{cor:q to 1}.  
We refer to these roots as the Orlik-Solomon exponents of $J$ since they coincide with the usual exponents of $W$ when $J$ is the empty set.  

Saito \cite{saito} introduced the notion of a free hyperplane arrangement, which implies that the characteristic polynomial of the arrangement factors into linear factors with positive roots.   It was finally resolved by Douglass \cite{douglass} and Broer \cite{broer:decomp} in a uniform manner that the restricted hyperplane arrangements $\mathcal{A}^J$ are all free, which explains the factorization of $\chi_J(x)$. 

\subsection{\texorpdfstring{The $t^n$-representation $\H$}{The tn-representation}} 

We explore a family of graded representations $\H$ of $W$ parametrized by certain natural numbers $t$.
This family of representations
arises in the theory of rational Cherednik algebras and is closely connected to Haiman's work on the diagonal harmonics.
%\marginpar{\tiny{Haiman ICM}}
The ungraded version of this 
family of representations also appears in the context of Springer theory for the affine Weyl group of $W$ and has a connection to the hyperplane arrangements $\mathcal{A}^J$ in the previous section.
% \cite{sommers:affineweylgp}.  
These various connections allow us to explain the coincidence of the occurrences of $V$ in the (usual) Springer theory with the Orklik-Solomon exponents.

Let $L^{\vee}$ denote the coroot lattice with respect to a maximal torus $T$ of $G$.
Let $f$ be the index of $L^{\vee}$ in the lattice of coweights. 
Recall that a bad prime for $W$ is one which divides a coefficient of a highest root for $W$. 
We say that $t$ is very good for $W$ if it is prime to all bad primes and also to $f$.   The latter condition only affects type $A_n$ (where $f=n+1$) since in other types the prime divisors of $f$ are bad primes.

Let $h$ be the Coxeter number of $W$.   
The following list summarizes when $t$ is very good for $W$:
\begin{enumerate}
%\item in type $A_n$, $t$ is prime to $h = n+1$ 
\item In types $B_n, C_n, D_n$, $t$ is very good if and only if $t$ is odd.
\item In other types, $t$ is very good if and only if $t$ is prime to $h$.
\end{enumerate}

The family of graded representations we are interested in is defined as follows.
Let $t$ be any natural number.
Consider the element $\H$ of $R(W)$ defined by
$$\H = \sum_{i=0}^n (-1)^i q^{it}S \otimes \wedge^i V.$$
Then $\H$ can be considered as a graded virtual representation of $W$ with $q$ in degree $1$.
For $w \in W$, the graded trace of $w \in W$ on $\H$ is given by 
\begin{equation}  \label{graded_trace}
\frac{p_w(q^t)}{ p_w(q)}
\end{equation}
where $p_w(x)$ denotes the characteristic polynomial of $w$ acting on $V$
(see  \cite{solomon}, \cite{gordon:diagonal}, or \cite{beg}).   This holds for all $t$.  

If $t$ is a value for which $\H$ is an actual finite-dimensional representation,
then the graded trace is a polynomial in $q$.  Therefore the ungraded trace of 
$w$ on $\H$ is the value of this polynomial at $q=1$, which can
be obtained by continuity by
letting $q \to 1$ in Equation \ref{graded_trace}.
Therefore the ungraded trace is given by $$t^{d(w)}$$
where $d(w)$ is the number of eigenvalues equal to $1$ in the action of $w$ on $V$.  
In particular, $\H$ is a representation of dimension $t^n$. % ($n$ recall is the rank of $W$).

There is also a representation $S_t$ of $W$ coming from the permutation action of $W$ on $L^{\vee}/t L^{\vee}$. 
If $t$ is very good, then the trace of $w$ on $S_t$ is also given
by $t^{d(w)}$ and 
$S_t$ has a decomposition into induced representations:
\begin{equation} \label{decomp}
S_t \simeq \sum_{J \subset \Pi} f_J(t) \Ind_{W_J}^W (\complex)
\end{equation}
where the sum is over a set of representatives for the $W$-orbits on subsets of $\Pi$.
Moreover $f_J(x)$ is related to the characteristic polynomial of the corresponding restricted hyperplane arrangement $\mathcal{A}^J$:
$$f_J(x) = \frac{1}{|W^J|} \chi_J(x)$$ with
$$W^J := N_W(W_J)/W_J.$$  
For a discussion of these results and another way to compute the Orlik-Solomon exponents, see  \cite{sommers:affineweylgp}.

The following proposition summarizes the properties of $\H$ that we need.
\begin{proposition} \label{properties}
\begin{enumerate}
\item $\H$ is an actual finite-dimensional representation if and only if $t$ is very good.  
\item If part (1) holds, then $\H$ is a quotient of $S=S^*(V)$ by an ideal generated by a copy of $V$ in degree $t$.
\item If part (1) holds, then $\H \simeq S_t$ as ungraded representations of $W$.
\end{enumerate}
\end{proposition}

\begin{proof}
This is a compilation of known results.  In the classical groups it was shown by Haiman \cite{haiman} in type $A$ and Gordon \cite{gordon:diagonal} in $B, C, D$ by a direct computation that if $t$ is very good then $\H$ is a true finite-dimensional representation and also that part (2) holds.  This could also be shown in all types by using the theory of rational Cherednik algebra and various cases were handled by Gordon \cite{gordon:diagonal}, Berest-Etingof-Ginzburg \cite{beg}, and Varagnolo-Vasserot \cite{vv} and Etingof \cite{etingof}.  If $t$ is very good, then clearly part (3) is true by the previous discussion.  

It remains to resolve that $\H$ is finite-dimensional only if $t$ is very good.  This could be deduced from \cite{vv} or \cite{etingof} or in the following alternative manner.  If $\H$ is finite-dimensional then the decomposition of Equation \ref{decomp} would still be true since it holds for an infinite number of $t$.  Then $f_J(t)$ would be a nonnegative integer for all $J$.  This can be shown to fail if $t$ is not very good by a simple case-by-case argument (which we omit).
\end{proof}

%In work of Gordon \cite{gordon:diagonal} and 
%$\H$ is shown to be an actual finite-dimensional (graded) representation of $W$ when 
%$t = ah +1$ with $a \in \zz_{\geq 0}$ or $t$ is odd in types $B_n$ or $C_n$ or $D_n$.
%Later $\H$ was shown to be an actual representation whenever $t$ is prime to $h$ in \cite{etingof:notes}, \cite{vv}.
%In summary 
%see these references for the full list of $t$ for which $\H$ is an actual representation.
%Moreover, in these cases $\H$ is isomorphic to a quotient of $S$ by a copy of $V$ in degree $t$.
%We will not use the full structure of the rational Cherednik algebra here, only the structure of $\H$ as a representation of $W$.
%

\subsection{\texorpdfstring{Writing $\H$ in terms of the $Q_{e, \phi}$}{Decomposing H into Springer representations}}

From now on assume that $t$ is very good, so that $\H$ is an actual finite-dimensional representation.
We will now decompose $\H$ into a sum of $Q_{e, \phi}$, with coefficients that will turn out to be polynomials in $q$.  
This decomposition will be a $q$-analog of the decomposition in Equation \ref{decomp}.
 
Proceeding as in Section \ref{section:shoji},
divide both sides of Equation \ref{equation:1} by $|G^F|$, then multiply the result by $(-q^t)^j$, and sum up over $j$
to obtain the identity in $R(W)$:

\begin{equation} \label{Main_equation_3}
\H =   \sum_{e}  \sum_{\phi} f_{e, \phi}(q;t) Q_{e, \phi}
\end{equation}
where
\begin{equation*}
f_{e, \phi}(q;t) := (-1)^n \sum_c  \sum_{\phi'} \left[ \sum^n_{j=0} \langle Q_{e, \phi'}, \wedge^{n-j} V \rangle  (-q^{t})^j \phi'(c) \right] \frac{ \phi(c)}{|Z_{G^F}(e_c)|}
\end{equation*}
We have used $$| \orbit_{e_c} | = \frac{|G^F|}{|Z_{G^F}(e_c )|}.$$
We consider the $f_{e, \phi}(q;t)$ only when $\phi \in \ar$ occurs in the Springer correspondence. 
%since  \cite{shoji:green_classical} $Q_{e,\phi}$ is nonzero if and only if  $\phi$ occurs in the Springer correspondence.

We use the version of Theorem \ref{weaker_theorem} in Equation \ref{theorem:variant} to simplify $f_{e, \phi}$:
\begin{equation*}
f_{e, \phi} = q^{t(n - d -s +1)} \prod_{j=1}^{s-1} (q^{t} - q^{ m_j}) 
\big( \sum_{i=0}^d (-1)^{d-i} q^{it + (d-i)m_s}  \sum_c  \frac{ \!\wedge^{d-i}  \pi_s(c) \phi(c)} {|Z_{G^F}(e_c)|} \big)
%\sum_c \big[ (-1)^{d+1}(q^{td} - q^{t(d-1)+m_s}\pi_s(c) + q^{2m_s}\!\wedge^2 \! \pi_s + \cdots + y^d q^{dm_s} \! \wedge^d \! \pi_s)  \big] \frac{ \phi(c)}{|Z_{G^F}(e_c)|}
\end{equation*}

Let $\tilde{s} = d+s-1$ and $m = \sum_{j=0}^{s-1} m_j + dm_s$.  
%Then if  $$C(q) = \langle Q_e, V \rangle_W$$ evaluated at the identity of $A(e)$,
%we have  $\tilde{s} = C(1)$ and $m = C'(1)$.
Then factoring out some powers of $q$,
\begin{equation}  \label{formula:f}
f_{e, \phi} = q^{t(n - \tilde{s})+ m} \prod_{j=1}^{s-1} (q^{t-m_j} - 1)
\big( \sum_{i=0}^d (-1)^{d-i} q^{i(t - m_s)}  \sum_c  \frac{ \!\wedge^{d-i}  \pi_s(c) \phi(c)} {|Z_{G^F}(e_c)|} \big)
\end{equation}

%$d=2$
%$d=2$ in F4  $q^{2t} - q^{t+m_s}? + q^{2m_s}?$
%$d=3$ in E8  $q^{3t} - q^{2t+m_s}? + q^{t + 2m_s}? + q^{3m_s}$ 

%%Indeed, we essentially did this more generally with the polynomials $\ttau(\chi)$ from Section \ref{section:shoji}. 

\begin{proposition}  \label{poly}
$f_{e, \phi}(q; t)$ is a polynomial in $q$.
\end{proposition}

\begin{proof}
By induction on the dimension of $\orbit_e$.   
%As noted above, $f_{e, \phi}$ is only defined when 
%$\phi$ occurs in the Springer correspondence and so $\chi_{e,\phi}$ is nonzero.
Take the inner product of both sides 
of Equation \ref{Main_equation_3} with respect to $\chi = \chi_{e, \phi}$:
$$\langle \H, \chi \rangle  =  f_{e, \phi} q^{d(e)} + \sum_{e'} \sum_{\phi'} f_{e', \phi'} \langle Q_{e', \phi'}, \chi \rangle$$

On the right-hand side, we have used \cite{borho-mac} as we did for Equation \ref{Main_equation_2}:  
first that $\chi_{e,\phi}$ only occurs in $Q_{e, \phi'}$ in the top degree and then only for $\phi = \phi'$ 
(the latter is already true by the Springer correspondence); 
and second 
that the sum is over representatives $e'$ of orbits properly contained in the closure of the orbit through $e$.

The left-hand side is a polynomial in $q$ by the assumption on $t$.  Since $\H$ is a quotient of $S$ 
from Proposition \ref{properties},
the lowest possible power of $q$ in this polynomial is the lowest power 
appearing in the fake degree of $\chi_{e,\phi}$, and this power 
is greater than or equal to $d(e)$, again by \cite{borho-mac}. 
Thus the claim is true then for $e=0$.

For $e \neq 0$, induction implies that the terms $f_{e', \phi'}$ are polynomial in $q$.
The polynomial $\langle Q_{e', \phi'}, \chi \rangle$ is divisible by $q^{d(e)}$ \cite{borho-mac}. 
Since $q^{d(e)}$ divides $\langle \H, \chi \rangle$, 
it follows that $f_{e, \phi}(q;t)$ is a polynomial in $q$ since we can divide both sides by $q^{d(e)}$ and express 
$f_{e, \phi}(q;t)$ as a sum of polynomials in $q$.
\end{proof}

For $J \subset \Pi$, let $\levi_J$ be the Levi subalgebra corresponding to $J$.

\begin{theorem}  \label{q to 1}
When $e$ is regular in $\levi_J$, then 
$$f_{e, \phi}(1; t) = f_J(t)$$
for any $\phi \in \ar$ appearing in the Springer correspondence.

When $e$ is not regular in a Levi subalgebra,
$f_{e, \phi}(1; t) = 0$.
\end{theorem}

\begin{proof}
Since the $f_{e, \phi}(q; t)$ are polynomial in $q$, we can set $q=1$ on both sides 
in Equation \ref{Main_equation_3} and obtain a decomposition:
$$S_t = \sum_{e}  \sum_{\phi} f_{e, \phi}(1;t) Q_{e, \phi}\big|_{q=1}$$
where we have used that 
$\H \simeq S_t$ as ungraded representations
by Proposition \ref{properties}.

Next we need Lusztig's result about induction of Springer representations \cite{lusztig:induction}:
if $e \in \levi_J$, then 
$H^*(\flag_e) \simeq \Ind_{W_J}^W H^*(\flag^J_e)$ as $W$-representations,
where $H^*(\flag^J_e)$ is the Springer representation of $W_J$ for $e$ with respect to $\levi_J$.
In particular if $e$ is regular in $\levi_J$, then $H^*(\flag^J_e)$ is the trivial representation of $W_J$ and thus
\begin{equation} \label{gyuri_induct}
H^*(\flag_e) \simeq \Ind_{W_J}^W( \complex).
\end{equation}
Recall that by definition $\sum_{\phi \in \ar}  (\dim \phi) Q_{e, \phi} = H^*(\flag_e)$ as  a $W$-representation.
Therefore the decomposition of $S_t$ from Equation \ref{decomp} is also a decomposition of 
$S_t$ into a sum of ungraded $Q_{e,\phi}$'s, with each induced factor corresponding to $H^*(\flag_e)$ for $e$ regular in $\levi_J$. 
That is,
$$S_t = \sum_{e}  \sum_{\phi} f_J(t)(\dim \phi) Q_{e, \phi}\big|_{q=1},$$
where the outer sum runs over a set of representatives $e$ of the nilpotent orbits for which $e$ is regular 
in some $\levi_J$.

Now the $Q_{e, \phi}$'s form a $\qq(q)$-basis of $R(W)$. 
%any graded representation of $W$ with a decomposition into Green functions $Q_{e, \phi}$ must be unique.  
Indeed this is a consequence of the previously used results of Borho-MacPherson \cite{borho-mac}:  namely, the transition matrix $\langle Q_{e,\phi}, \chi_{e', \phi'} \rangle$ is upper triangular for any total ordering respecting the partial order on nilpotent orbits.  
In fact the transition matrix takes values in $\zz[q]$ and the diagonal entries are powers of $q$,
so by setting $q=1$ in the transition matrix, it follows that the 
$$Q_{e, \phi}\big|_{q=1}$$ form a $\zz$-basis for the representation ring of $W$ with coefficients in $\zz$.  

Therefore the two decompositions of $S_t$ coincide.  This means that 
$f_{e, \phi}(1;t) =0$ if $e$ is not regular in a Levi subalgebra. 
%since the corresponding Green functions 
%do not appear in the decomposition of $S_t$ from Equation \ref{decomp}.
On the other hand, if $e$ is regular in $\levi_J$, 
then $$f_{e, \phi}(1;t) =  (\dim \phi) f_J(t)$$ for any $\phi \in \ar$.  
Noting that $A(e)$ is elementary abelian when $e$ is regular in a Levi subalgebra (and thus $\dim \phi =1$) concludes the proof.
\end{proof}

\begin{corollary} \label{cor:q to 1}
Let $e$ be regular in a Levi subalgebra $\levi_J$.
Then the characteristic polynomial of the restricted arrangement $\mathcal{A}^J$ in $V^{W_J}$ factors as 
$$\chi_J(x)  = \prod_{j=1}^{s} (x - m_j)$$
where $m_1, m_2, \dots, m_s$ are the graded occurrences of $V$ in $H^*(\flag_e)$.
\end{corollary}

\begin{proof}
By the theorem and the fact that $$f_J(t) = \frac{1}{|W^J|} \chi_J(t),$$ it will be enough to show that 
$f_{e, \phi}(1;t) = C \prod_{j=1}^{s} (t - m_j)$ where $C$ is some constant since $\chi_J(x)$ is monic by definition.

In the previous two proofs we did not use the explicit formula for $f_{e, \phi}(q;t)$ coming from Theorem \ref{weaker_theorem}, but we use it here.
Since $\pi_s$ is one-dimensional, Equation \ref{formula:f} becomes
\begin{equation} \label{one_d}
f_{e, \phi} = q^{t(n - s)+ m} \prod_{j=1}^{s-1} (q^{t-m_j} - 1)
 \sum_c  \frac{ (q^{t-m_s} - \pi_s(c)) \phi(c)} {|Z_{G^F}(e_c)|}.
 \end{equation}

By Equation \ref{gyuri_induct} and Frobenius reciprocity, the multiplicity $s$ of $V$  in $H^*(\flag_e)$
equals $n - |J|$ and therefore $s$ is both the degree of $\chi_J(x)$ and the rank of $Z_G(e)$.

Recall the assumption from Section \ref{section:shoji} that $e$ is rational and split.
Since the rank of $Z_G(e)$ equals $s$, 
the Chevalley-Steinberg formula for  $|Z_{G^F}(e)|$ implies that
$(q-1)^s$ divides $|Z_{G^F}(e)|$ and no higher power of $(q-1)$ divides.  
When $e$ is regular in $\levi_J$, we observe that for $c \in A(e)$ nontrivial, 
the highest  possible power of $(q-1)$ dividing $|Z_{G^F}(e_c)|$ is $s-1$; moreover, 
in the cases where $\pi_s$ is nontrivial, 
the highest possible power of $(q-1)$ dividing $|Z_{G^F}(e_c)|$ is only $s-2$ when $c$ is nontrivial.

Taking the limit $q \to 1$ in Equation \ref{one_d}, the only term that survives in the sum
is for $c=1$ and we get
$f_{e, \phi}(1;t) = C \prod_{j=1}^{s} (t - m_j)$ for some constant $C$, which establishes the result.  
%Since $f_{e, \phi}(1;t)$ is equal to $\chi_J(t)$ times a constant, 
%t follows that $q=1$ is a root of each of the $s$ factors in Equation \ref{one_d}.
\end{proof}

\begin{remark}
Certainly $C=\frac{1}{|W^J|}$.  From the proof of the corollary it therefore follows that
$$\lim_{q \to 1} \frac{(q-1)^s}{|Z_{G^F}(e)|} = \frac{1}{|W^J|}.$$  This could also be deduced using \cite{brundan-goodwin:poly} where it
is shown that the Weyl group of $Z_{G}(e)$ is isomorphic to $W^J$.
\end{remark}

\begin{remark}
Theorem \ref{q to 1} shows that when $e$ is regular in $\levi_J$ that 
$f_{e, \phi}(q;t)$ is a $q$-analog of 
$$f_J(t) = \frac{1}{|W^J|} \chi_J(t).$$  The  computation of $f_{e, \phi}(q;t)$  will be the subject of a sequel paper with Reiner.  The paper \cite{bessis-reiner:cyclic} was a motivation for the present paper, especially an earlier version which asked whether such $q$-analogs existed.
\end{remark}

\begin{remark}
Corollary \ref{cor:q to 1} gives another proof that the characteristic polynomials $\chi_J(x)$ factor.
One can wonder, 
since Conjecture \ref{conj1} implies Theorem \ref{weaker_theorem} which in turn implies that $\chi_J(x)$ factors, 
whether the conjecture also has something to do with 
the fact that the restricted arrangement $\mathcal{A}^J$ is free.
\end{remark}

\begin{remark}
The representations $\H$ are related to the work of Haiman \cite{haiman:cdm} on the diagonal harmonics and various generalizations.  
By \cite{gordon:diagonal},  \cite{beg},  \cite{gordon-stafford}, 
$\H \otimes \epsilon$ where $\epsilon$ is the sign representation is isomorphic to a natural quotient of the polynomials on two copies of $V$, after 
replacing the second grading variable by $q^{-1}$
and shifting by an appropriate power of $q$.  Of particular interest is the case
when $t=h+1$ since this connects with the 
diagonal harmonics.    See the above mentioned papers for a more careful description.
\end{remark}

\begin{remark}
The decomposition of $\H$ into a sum of $Q_{e,\phi}$'s in Equation \ref{Main_equation_3}
should be compared to the decomposition of the total homology of certain fixed-point varieties in the affine flag manifold 
in \cite{sommers:affineweylgp}.  A connection along these line was already noticed in \cite{beg}.
In the situation of \cite{sommers:affineweylgp}, however, the representation was only determined as an ungraded representation.  
The present work suggests that the graded representation there should agree with the the other 
natural singly-graded representation obtained from the functions on two copies of V (see \cite{haiman:cdm}).  In particular, in type $A_n$ the affine Springer representations of \cite{sommers:affineweylgp} should be equal to the (graded) parking function module for $t=h+1$ and its generalizations for other $t$.
Evidence for this conjecture is provided by the fact that the combinatorics of the extended Shi arrangement plays a parallel role in both the homology of the affine Springer fiber and the parking function module.
\end{remark}

\bibliography{exterior}
\bibliographystyle{pnaplain}

\end{document}